\documentclass
{amsart}
\usepackage{graphicx}
\newtheorem{thm}{Theorem}[section]
\theoremstyle{definition}
\newtheorem{cor}[thm]{Corollary}
\newtheorem{prop}[thm]{Proposition}
\newtheorem{defn}[thm]{Definition}
\newtheorem{lem}[thm]{Lemma}

\newtheorem{rem}[thm]{Remark}
\newtheorem{ex}[thm]{Example}

\numberwithin{equation}{section}
\begin{document}
\title[The duals of annihilator conditions for modules]
{The duals of annihilator conditions for modules}\footnotetext{This research was supported with a grant from Farhangian University}
\author{Faranak  Farshadifar}
\address{Department of Mathematics, Farhangian University, Tehran, Iran.}
\email{f.farshadifar@cfu.ac.ir}


% ----------------------------------------------------------------
\begin{abstract}
Let $R$ be a commutative ring with identity and let $M$ be an $R$-module.
The purpose of this paper is to introduce and investigate the submodules of an $R$-module $M$ which satisfy the dual of  Property $\mathcal{A}$,  the dual of strong  Property $\mathcal{A}$,  and the dual of proper strong  Property $\mathcal{A}$. Moreover, a submodule $N$ of $M$ which satisfy Property $\mathcal{S_J(N)}$ and Property $\mathcal{I^M_J(N)}$ will be introduced and investigated.
\end{abstract}
\maketitle

\section{Introduction}
\noindent
Throughout this paper, $R$ will denote a commutative ring with
identity and $\Bbb Z$ will denote the ring of integers.

Let $M$ be an $R$-module. The set of zero divisors of $R$
on $M$ is $Z_R(M)= \{r \in R| rm=0 \ for \ some\ nonzero\ m \in M\}$ and the set of torsion elements
of $M$ with respect to $R$ is $T_R(M) = \{m \in M |rm = 0 \ for \ some\ 0 \not= r \in R\}$.

An $R$-module $M$ satisfies \textit{Property $\mathcal{A}$} (resp., \textit{Property $\mathcal{T}$}) if for every
finitely generated ideal $I$ of $R$ (resp., finitely generated submodule $N$ of $M$)
with $I \subseteq Z_R(M)$ (resp., $N \subseteq T_R(M)$), there exists a nonzero $m \in M$
(resp., $r \in R$) with $Im = 0$ (resp., $rN = 0$), or equivalently
$(0:_MI) \not =0$ (resp., $Ann_R(N) \not=0$) \cite{MR3661610}.
An $R$-module $M$ satisfies \textit{strong  Property $\mathcal{A}$} (resp., \textit{strong Property $\mathcal{T}$}) if
for any $r_1, . . . , r_n\in Z_R(M)$ (resp., $m_1, . . .,m_n \in T_R(M)$), there exists a non-zero $m\in M$ (resp., $r \in R$) with
$r_1m =\cdots = r_nm = 0$ (resp., $rm_1 =\cdots = rm_n = 0$)  \cite{MR3661610}.
An $R$-module $M$ satisfies \textit{proper strong  Property $\mathcal{A}$}
if for any proper finitely generated ideal $I = \langle a_1, a_2, . . . , a_n\rangle $
of $R$ such that $a_i \in Z_R(M)$ we have $(0:_MI) \not =0$ \cite{MR4160990}. The
class of modules satisfies proper strong  Property $\mathcal{A}$ lying properly between the class of modules
satisfies strong Property $\mathcal{A}$ and Property $\mathcal{A}$ \cite[Corollary 2.12]{MR4160990}.

Let $M$ be an $R$-module. The subset $W_R(M)$ of $R$ (that is the dual notion of $Z_R(M)$) is defined by $\{r \in R| rM \not =M\}$ \cite{Ya98}.
A non-zero submodule $N$ of  $M$ is said to be \emph{secondal} if $W_R(N)$ is an
ideal of $R$. In this case, $W_R(N)$ is a prime ideal of $R$ \cite{MR2821719}.

Recently, the annihilator conditions on modules over commutative rings have attracted the attention of several researchers. A brief history of this can be found in \cite{MR3661610, MR4160990}.
The purpose of this paper is to introduce and study the  dual of  Property $\mathcal{A}$, the dual of strong  Property $\mathcal{A}$, and the dual of proper strong Property $\mathcal{A}$ for modules over a commutative ring. Also, for a submodule $N$ of an $R$-module $M$ we introduce and investigate the Properties $\mathcal{S_J(N)}$ and $\mathcal{I^M_J(N)}$.
Some of the results in this article are dual of the results for
Property $\mathcal{A}$, strong  Property $\mathcal{A}$, and proper strong Property $\mathcal{A}$
considered in \cite{MR4160990} and \cite{MR3661610}.
%%%%%%%%%%%%%%%%%%%%%
%%%%%%%%%%%%%%%%%%%%%
%%%%%%%%%%%%%%%%%%%%%
%%%%%%%%%%%%%%%%%%%%%
%%%%%%%%%%%%%%%%%%%%%
%%%%%%%%%%%%%%%%%%%%%
\section{The duals of  Property $\mathcal{A}$ and strong  Property $\mathcal{A}$  for modules}
\begin{defn}\label{d1.1}
We say that an $R$-module $M$ satisfies the \textit{dual of  Property $\mathcal{A}$}  if for each finitely generated ideal $I$ of $R$ with
$I \subseteq W_R(M)$ we have $IM \not =M$.
\end{defn}

A proper submodule $N$ of an $R$-module
$M$ is said to be \emph{completely irreducible} if $N=\bigcap _
{i \in I}N_i$, where $ \{ N_i \}_{i \in I}$ is a family of
submodules of $M$, implies that $N=N_i$ for some $i \in I$. It is
easy to see that every submodule of $M$ is an intersection of
completely irreducible submodules of $M$ \cite{FHo06}.

\begin{defn}\label{d21.1}
We say that an $R$-module $M$ satisfies \textit{the dual of strong  Property $\mathcal{A}$}  if
for any $a_1, . . . , a_n \in W_R(M)$,
there exists a completely irreducible submodule $L$ of $M$ such that  $a_iM\subseteq L \not=M$ for $i=1,2,..,n$.
\end{defn}

Clearly, if an $R$-module $M$ satisfies the dual of strong Property $\mathcal{A}$, then $M$ satisfies the dual of Property  $\mathcal{A}$. Nevertheless, the following example shows that the converse is not true in general.
\begin{ex}\label{e121.1}
The $\Bbb Z$-module $\Bbb Z$ satisfies the dual of Property  $\mathcal{A}$ but does not satisfies the dual of strong Property $\mathcal{A}$.
\end{ex}

\begin{rem}\label{r2.1}
Let $N$ and $K$ be two submodules of an $R$-module $M$. To prove $N\subseteq K$, it is enough to show that if $L$ is a completely irreducible submodule of $M$ such that $K\subseteq L$, then $N\subseteq L$ \cite{MR2821719}.
\end{rem}
\begin{thm}\label{tt1.5}
Let $M$ be an $R$-module. Consider the conditions:
\begin{itemize}
\item [(a)] $M$ satisfies the dual of strong  Property $\mathcal{A}$;
\item [(b)] $M$  is a secondal $R$-module.
\end{itemize}
Then $(a)\Rightarrow (b)$. If further $R$ is a $PID$, then $(b)\Rightarrow (a)$.
\end{thm}
\begin{proof}
 $(a)\Rightarrow (b)$.
 Let $a , b \in W_R(M)$. Then by part (a),  there exists a completely irreducible submodule $L$ of $M$ such that $aM\subseteq L \not=M$ and $bM\subseteq L \not=M$. Thus $(a-b)M \subseteq L \not=M$ and so $a-b \in W_R(M)$. This implies that  $M$  is a secondal $R$-module.

$(b)\Rightarrow (a)$.
Let $a_1, . . . , a_n \in W_R(M)$. Then by part (b), $\langle a_1, . . . , a_n\rangle$ is an ideal of $R$. As $R$ is a $PID$, there exists an $a \in R$ such that $\langle a_1, . . . , a_n\rangle=Ra$. Thus $a \in W_R(M)$. Hence there exists a completely irreducible submodule $L$ of $M$ such that $aM \subseteq L \not=M$ by Remark \ref{r2.1}. This implies that  $a_iM\subseteq L \not=M$ for $i=1,2,..,n$, as needed.
\end{proof}

\begin{thm}\label{t1.3}
Let $f : R \rightarrow \acute{R}$ be a homomorphism of commutative rings and
let $M$ be an $\acute{R}$-module. Consider $M$ as an $R$-module with $rm := f(r)m$ for
$r \in R$ and $m \in M$.
\begin{itemize}
\item [(a)] Suppose for each (finitely generated) ideal $I$ of  $R$, $f(I)\acute{R} = \{f(i)\acute{r} | i \in I, \acute{r} \in \acute{R} \}$ (e.g., $f$ is surjective or $f : R \rightarrow R_N$, $f(r) = r/1$, where $N$ is
a multiplicatively closed subset of $R$). Then $M$ satisfies  the dual of  Property $\mathcal{A}$   as an
$\acute{R}$-module implies $M$ satisfies  the dual of  Property $\mathcal{A}$   as an $R$-module.
\item [(b)]  Suppose that every (finitely generated) ideal $J$ of $\acute{R}$ has the form $J = f(I)\acute{R}$
for some (finitely generated) ideal $I$ of $R$ (e.g., $f$ is surjective or $f : R \rightarrow
R_N$, $f(r) = r/1$, where $N$ is a multiplicatively closed subset of $R$). Then
$M$ satisfies  the dual of  Property $\mathcal{A}$   as an $R$-module implies $M$ satisfies  the dual of  Property $\mathcal{A}$   as
an $\acute{R}$-module.
\end{itemize}
\end{thm}
\begin{proof}
(a) Suppose $M$ satisfies  the dual of  Property $\mathcal{A}$  as an $\acute{R}$-module. Let $I$ be an ideal
of $R$ with $I \subseteq W_R(M)$. So for $i \in I$, there is a $m \in M\setminus iM=M\setminus f(i)M$. Hence,
$f(i) \in W_{\acute{R}}(M)$ and so $\acute{r}f(i) \in W_{\acute{R}}(M)$ for each $\acute{r} \in \acute{R}$. Thus $f(I)\acute{R} = \{f(i)\acute{r} | i \in
I, \acute{r} \in \acute{R}\}$ is an ideal of $\acute{R}$ with $f(I)\acute{R} \subseteq W_{\acute{R}}(M)$. Suppose that $I$ is finitely generated.
Then $f(I)\acute{R}$ is finitely generated. Hence there is a $ m \in M\setminus f(I)\acute{R}M$.
This implies that $ m \in M\setminus IM$. Thus $M$ satisfies the dual of  Property $\mathcal{A}$   as an $R$-module.

(b) Suppose that $M$ satisfies the dual of  Property $\mathcal{A}$   as an R-module. Let J be an ideal
of $\acute{R}$ with $J \subseteq W_{\acute{R}}(M)$. Then there is an ideal $I$ of $R$ with $J = f(I)S$. For $i \in I$,
$f(i) \in W_{\acute{R}}(M)$. So, there is a $m \in M\setminus f(i)M=M\setminus iM$. So, $I \subseteq W_R(M)$. If
$J$ is finitely generated, we can choose $I$ to be finitely generated. Since $M$ satisfies
the dual of  Property $\mathcal{A}$   as an $R$-module, there is a $m \in M\setminus IM$. It follows that $m \in M\setminus f(I)\acute{R}M=M\setminus JM$. So, $M$ satisfies the dual of  Property $\mathcal{A}$   as an $\acute{R}$-module.
\end{proof}

\begin{cor}\label{t31.3}
Let $M$ be an $R$-module, $J \subseteq Ann_R(M)$ an ideal of $R$, and put $\bar{R} = R/J$. Then $M$ satisfies the dual of  Property $\mathcal{A}$ as an
$R$-module if and only if $M$ satisfies the dual of  Property $\mathcal{A}$   as an $\bar{R}$-module. In particular, $M$ satisfies the dual of  Property $\mathcal{A}$
as an $R$-module if and only if $M$ satisfies the dual of  Property $\mathcal{A}$   as an $R/Ann_R(M)$-module.
\end{cor}
\begin{proof}
 This follows from Theorem \ref{t1.3}.
\end{proof}

Recall that an $R$-module $M$ is said to be \emph{Hopfian} (resp. \emph{co-Hopfian})
if every surjective (resp. injective) endomorphism $f$ of $M$ is an isomorphism.

An $R$-module $M$ is said to be a \emph{multiplication module} if for every submodule $N$ of $M$ there exists an ideal $I$ of $R$ such that $N=IM$ \cite{Ba81}.

A submodule $N$ of an $R$-module $M$ is said to be \emph{idempotent} if $N=(N:_RM)^2M$. Also, $M$ is said to be \emph{fully idempotent}
if every submodule of $M$ is idempotent \cite{AF122}.

An $R$-module $M$ is said to be a \emph{comultiplication module} if for every submodule $N$ of $M$ there exists an ideal $I$ of $R$ such that $N=(0:_MI)$ \cite{AF07}.  $R$ is said to be a \textit{comultiplication
ring} if, as an $R$-module, $R$ is a comultiplication $R$-module \cite{AF08}.

A submodule $N$ of an $R$-module $M$
is said to be \emph{coidempotent} if $N=(0:_MAnn_R(N)^2)$. Also,
 an $R$-module $M$ is said to be \emph{fully coidempotent}
if every submodule of $M$ is coidempotent \cite{AF122}.

\begin{prop}\label{l1.2}
Let $M$ be an $R$-module. Then we have the following.
\begin{itemize}
  \item [(a)] If $R$ is a comultiplication ring and $M$ is a faithful $R$-module, then $M$ satisfies  Property $\mathcal{A}$  and the dual of  Property $\mathcal{A}$.
  \item [(b)] If $M$ is a Hopfian comultiplication (in particular, $M$ is a fully coidempotent) $R$-module and satisfies the dual of  Property $\mathcal{A}$, then $M$ satisfies  Property $\mathcal{A}$.
  \item [(c)] If $M$ is a co-Hopfian multiplication (in particular, $M$ is a fully idempotent) $R$-module and satisfies Property $\mathcal{A}$, then $M$ satisfies the dual of  Property $\mathcal{A}$.
  \item [(d)] If $R$ is a principal ideal ring, then $M$ satisfies the dual of strong  Property $\mathcal{A}$.
\end{itemize}
\end{prop}
\begin{proof}
(a) This follows from \cite[Lemma 3.11]{AF08}.

(b) First note that every fully coidempotent $R$-module is a Hopfian comultiplication $R$-module by \cite[Theorem 3.9 and Proposition 3.5]{AF122}. As $M$ is a Hopfian comultiplication $R$-module, $Z_R(M)=W_R(M)$. Now the result follows from \cite[Proposition 3.1]{MR2475337}.

(c)  First note that every fully idempotent $R$-module is a co-Hopfian multiplication $R$-module by \cite[Proposition 2.7]{AF122}. Since $M$ is a co-Hopfian multiplication $R$-module, $Z_R(M)=W_R(M)$. Now the result follows from \cite[Note 1.13]{MR2096268}.

(d) This is clear.
\end{proof}

\begin{lem}\label{l1.2}
Let $S$ be a multiplicatively closed subset of $R$, $I$ and ideal of $R$, and $M$ be an $R$-module.
Then we have the following.
\begin{itemize}
  \item [(a)] If $S^{-1}I \subseteq W_{S^{-1}R}(S^{-1}M)$, then $I \subseteq W_R(M)$.
  \item [(b)] If $Z_R(M) \cap S=\emptyset$, $W_R(M) \cap S=\emptyset$, and $I \subseteq W_R(M)$, then $S^{-1}I \subseteq W_{S^{-1}R}(S^{-1}M)$.
  \item [(c)] If $M$ is an Hopfian module (in particular, $M$ is a multiplication or coidempotent module), $W_R(M) \cap S=\emptyset$, and $I \subseteq W_R(M)$, then $S^{-1}I \subseteq W_{S^{-1}R}(S^{-1}M)$.
\end{itemize}
\end{lem}
\begin{proof}
(a) This is clear.

(b) Suppose that $S^{-1}I \not\subseteq W_{S^{-1}R}(S^{-1}M)$ and seek for a contradiction. Then $S^{-1}(aM)=S^{-1}M$ for some $a \in I$. As  $I \subseteq W_R(M)$, there exists $m \in M \setminus IM$. Now we have $stm=sam_1$ for some $s, t \in S$ and $m_1 \in M$. Since  $W_R(M) \cap S=\emptyset$, $tM=M$ and so $m_1=tm_2$ for some $m_2 \in M$. Hence, $st(m-am_2)=0$. Now  $Z_R(M) \cap S=\emptyset$ implies that $m=am_2$, which is a contradiction.

(c) This follows from the fact that $Z_R(M) \subseteq W_R(M)$ and part (b).
\end{proof}

\begin{cor}\label{c1.5}
Let $S$ be a multiplicatively closed subset of $R$ and $M$ be an $R$-module. Consider the conditions:
\begin{itemize}
\item [(a)] $M_S$ satisfies  the dual of   Property $\mathcal{A}$  as an $R$-module;
\item [(b)] $M_S$ satisfies  the dual of   Property $\mathcal{A}$  as an $R_S$-module;
\item [(c)] $M$ satisfies the dual of   Property $\mathcal{A}$  as an $R$-module.
\end{itemize}
Then $(a)\Leftrightarrow (b)$. If further $S \cap W_R(M) = \emptyset$ and $S \cap Z_R(M) = \emptyset$, (a), (b) and (c) are equivalent.
\end{cor}
\begin{proof}
The equivalence of (a) and (b) follows from Theorem \ref{t1.3}. Now assume that $S \cap W_R(M) = \emptyset$ and $S \cap Z_R(M) = \emptyset$. Then $(b) \Leftrightarrow (c)$ from Lemma \ref{l1.2} (b).
\end{proof}

\begin{prop}\label{p24.۶}
Let $X$ be an indeterminate over $R$,  $M$ be an $R$-module, and $M[X]$ satisfies the dual of strong Property $\mathcal{A}$ over $R[X]$. Then $M$ satisfies the dual of Property $\mathcal{A}$.
\end{prop}
\begin{proof}
Let $I = \langle a_1,...,a_n\rangle$ be a finitely
generated ideal of $R$ such that $a_i \in W_R(M)$ for $i = 1,..., n$. Then $I[X] = \langle a_1,...,a_n\rangle R[X]$ is
a finitely generated ideal of $R[X]$ such that $a_i\in W_{R[X]}(M[X])$ for $i = 1,...,n$.
Since $M[X]$ satisfies the dual of strong Property $\mathcal{A}$ over $R[X]$, we get
$I[X]M[X] \not=M[X]$. This implies that $IM \not=M$. Thus
$M$ satisfies the dual of Property $\mathcal{A}$.
\end{proof}

Recall that a ring $R$ is called \textit{$B \acute{e} zout$} if every finitely generated ideal $I$ of $R$ is principal.

A submodule $N$ of an $R$-module $M$ is \textit{small} if for any submodule $X$ of $M$, $X + N = M$ implies that $X = M$.

A prime ideal $P$ of $R$ is said to be a \textit{coassociated prime ideal} of an $R$-module $M$ if there exists a cocyclic homomorphic image $T$ of $M$ such that $Ann_R(T) = P$.
The set of coassociated prime ideals of $M$ is denoted by $Coass(M)$ \cite{MR1481103}.

\begin{thm}\label{t1.6}
\begin{itemize}
\item [(a)] The trivial $R$-module vacuously satisfies  the dual of
 Property $\mathcal{A}$.
\item [(b)] Every module over a $B \acute{e} zout$ ring satisfies  the dual of  Property $\mathcal{A}$.
\item [(c)]  Let $R$ be a zero-dimensional commutative ring (e.g., $R$ is Artinian). Then every
$R$-module satisfies  the dual of  Property $\mathcal{A}$.
\item [(d)] Let $M$ be a finitely generated $R$-module. Then $M$ satisfies
 the dual of  Property $\mathcal{A}$.
\item [(e)] Let $M$ be an Artinian $R$-module. Then $M$ satisfies
 the dual of  Property $\mathcal{A}$. In fact, for any ideal $I$ of $R$ with $I \subseteq W_R(M)$, $IM \not = M$.
\item [(f)] Let $M$ and $\acute{M}$ be $R$-modules with $W_R(M) \subseteq W_R(\acute{M})$. If
$\acute{M}$ satisfies  the dual of  Property $\mathcal{A}$  (respectively, the dual of strong Property $A$), then $M \oplus \acute{M}$ satisfies
 the dual of  Property $\mathcal{A}$  (respectively, the dual of strong  Property $\mathcal{A}$).
\item [(g)] Let $N$ be a small submodule of $M$. Then $M$ satisfies
 the dual of Property $\mathcal{A}$ (resp. $M$ is a secondal module) if and only if $M/N$ satisfies
 the dual of  Property $\mathcal{A}$ (resp.  $M/N$ is a secondal module).
\end{itemize}
\end{thm}
\begin{proof}
(a) Note that $W_R(0) = \emptyset$.

(b) This is clear.

(c) Suppose $dim R = 0$ and $M$ is an $R$-module. We can assume $M\not = 0$. Let $I$ be
a finitely generated ideal of $R$ with $I \subseteq W_R(M)$. So, $I \subseteq P \subseteq W_R(M)$ for some prime
ideal $P$ of $R$ by using \cite[Theorem 2.15]{MR1481103}. Since $ht P = 0$ and $I$ is finitely generated, $I^n_P = 0$ for some $n\geq 1$.
Hence there is an $s \in R\setminus P$ with $I^ns = 0$. Since $s \in R \setminus P$, $s\not \in Ann_R(M)$. Thus $sM \not =0$. We have $I^nsM = 0$. Suppose $I^tsM\not= 0$, but $I^{t+1}sM = 0$. Then $IM \subseteq (0:_MI^ts) \not =M$. Therefore,  $IM \not =M$.

(d) Let $M$ be a finitely generated $R$-module and $I$ be an ideal of $R$ with $I \subseteq W_R(M)$.  Assume contrary that $IM=M$. Then $(1+a)M=0$ by \cite[Theorem 76]{MR0345945}. As $a \in I\subseteq W_R(M)$, there exists an $m \in M \setminus aM$. Now, $(1+a)m=0$  implies that $m \in aM$, which is a contradiction.

(e) As $M$ is an Artinian $R$-module, $W_R(M)=\cup^n_{i=1} P_i$ by using
\cite[Theorem 2.10 (c), Theorem 2.15, Corollary 3.2]{MR1481103}, where $P_i \in Coass(M)$.
Now let $I \subseteq W_R(M)$ be an ideal. Then $I \subseteq P_i$ for some  $P_i \in Coass(M)$. Hence for some completely irreducible submodule $L$ of $M$ with $L \not =M$, we have $I \subseteq P_i=(L:_RM)$. This implies that $IM \not=M$.

(f) Let $\acute{M}$ satisfies the dual of  Property $\mathcal{A}$. It is easy to see that $W_R(M \oplus \acute{M}) = W_R(M) \cup W_R(\acute{M}) = W_R(\acute{M})$. Let $I$ be a finitely generated
ideal of $R$ with $I \subseteq  W_R(M) \cup W_R(\acute{M}) = W_R(\acute{M})$. Then
$I\acute{M}\not =\acute{M}$. Thus there exists an $x \in I\acute{M}\setminus \acute{M}$. This implies that $(0,x) \not \in I(M \oplus \acute{M})$ and so  $I(M \oplus \acute{M})\not =M \oplus \acute{M}$, as needed.

(g) We always have $W_R(M/N)\subseteq W_R(M)$. As $N$ is small we get that $W_R(M)\subseteq W_R(M/N)$. Now the result is straightforward.
\end{proof}

Let $M$ be an $R$-module. The idealization $R(+)M =\{(a,m): a \in R, m \in  M\}$ of $M$ is
a commutative ring whose addition is component-wise and whose multiplication is defined as $(a,m)(b,\acute{m}) =
(ab, a\acute{m} + bm)$ for each $a, b \in R$, $m, \acute{m}\in M$ \cite{Na62}.
\begin{prop}\label{l1.7}
Let $M$ be an $R$-module. Then we have
$$
W_{R(+)M}(R(+)M)=W_R(R) (+)M.
$$

\end{prop}
\begin{proof}
First note that $W_R(M) \subseteq W_R(R)$.
Let $(a, x ) \in W_{R(+)M}(R(+)M)$. Then there exists $(b, y) \in R(+)M \setminus (a, x )(R(+)M)$. This implies that
for each $(c, z) \in R(+)M$, $(b, y) \not =(a,x)(c,z)$. Hence, $b \not =ac$ or $y \not =az+cx$. If  $b \not =ac$, then $a \in W_R(R)$ and we are done. If $y \not =az+cx$. Then by setting $c=0$, we have  $y \not =az$. Thus $a \in W_R(M) \subseteq W_R(R)$ and so
$W_{R(+)M}(R(+)M)\subseteq W_R(R)(+)M$. Now let $(a, x ) \in W_R(R)(+)M$. Then $a \in W_R(R)$. Thus there exist $r \in R$ such that $r \in R \setminus aR$.
Assume contrary that $(a,x)(R(+)M)=R(+)M$. Then $(r,0)=(a,x)(c,y)$ for some $(c,y) \in R(+)M$.  Thus $r=ac$,  which is a contradiction. Hence $(a,x)(R(+)M)\not =R(+)M$, as needed.
\end{proof}

\begin{ex}\label{e1.7}
Let $M = \oplus R/I$, where the sum runs over all proper finitely generated ideals of $R$. Then
 for a proper finitely generated ideal $I$ of $R$, $I \subseteq W_R(M)$ and $I(R/I)=\bar{0}\not =R/I$ implies that $M$
satisfies the dual of  Property $\mathcal{A}$. As $R$ is a submodule of $M$, we have  $R$ is a submodule of
an $R$-module satisfying the dual of  Property $\mathcal{A}$. Let $\acute{M}$ be any $R$-module. Then $M\oplus \acute{M}$ again,
satisfies the dual of  Property $\mathcal{A}$. Thus, any $R$-module is a submodule, homomorphic image, or
direct factor of a module satisfying the dual of  Property $\mathcal{A}$.
\end{ex}
%%%%%%%%%%%%%%%%%%%%%
%%%%%%%%%%%%%%%%%%%%%
%%%%%%%%%%%%%%%%%%%%%
%%%%%%%%%%%%%%%%%%%%%
%%%%%%%%%%%%%%%%%%%%%
%%%%%%%%%%%%%%%%%%%%%
\section{The duals of proper strong Property $\mathcal{A}$  for modules}

\begin{lem}\label{l1.8}
Let $M$ be an $R$-module and $S=R \setminus W_R(M)$. Then $S^{-1}R=R$ if and only if  $R=U(R) \cup W_R(M)$, where $U(R)$ is the set of all invertible elements of $R$.
\end{lem}
\begin{proof}
 Assume that $S^{-1}R=R$ and $s\in S$. Then $1/s \in S^{-1}R=R$  implies that $s$ is
invertible in $R$. Hence any element of  $S$ is invertible in $R$. The
converse is clear.
\end{proof}

\begin{defn}\label{d1.9}
We say that an $R$-module $M$ satisfies \textit{the dual of proper strong  Property $\mathcal{A}$} if for any proper finitely generated ideal $I = \langle a_1, a_2, . . . , a_n\rangle $
of $R$ such that $a_i \in W_R(M)$ we have $IM \not=M$.
\end{defn}

\begin{thm}\label{t1.10}
Let $M$ be an $R$-module. Then the following assertions are
equivalent:
\begin{itemize}
\item [(a)] $M$ satisfies the dual of proper strong  Property $\mathcal{A}$;
\item [(b)] $M$ satisfies the dual of  Property $\mathcal{A}$  and $\mathfrak{m} \cap W_R(M)$ is an ideal of $R$ for each maximal ideal $\mathfrak{m}$ of $R$.
\end{itemize}
\end{thm}
\begin{proof}
$(a) \Rightarrow (b)$
 Assume that  $M$ satisfies the dual of proper strong  Property $\mathcal{A}$. Clearly, $M$ satisfies the dual of  Property $\mathcal{A}$. Let $\mathfrak{m}$ be a
maximal ideal of $R$. Let $a, b \in \mathfrak{m} \cap W_R(M)$ and put $I = \langle a, b\rangle$ the ideal generated by $a$ and
$b$. Then $I\subseteq \mathfrak{m}$ and $a, b \in  W_R(M)$. Since  $M$ satisfies the dual of proper strong  Property $\mathcal{A}$, we get that $IM \not=M$.
It follows that $I\subseteq \mathfrak{m} \cap W_R(M)$ and thus $\mathfrak{m} \cap W_R(M)$ is an ideal of $R$.

$(b) \Rightarrow (a)$
Let $I = \langle a_1, a_2, . . . , a_n\rangle $ be a proper finitely generated ideal of $R$ such that $a_i \in W_R(M)$
for $i = 1, ... , n$. Let $\mathfrak{m}$ be a maximal ideal of $R$ with $I \subseteq \mathfrak{m}$. Then $a_1, a_2, ... , a_n\in\mathfrak{m} \cap W_R(M)$. As, by hypotheses, $\mathfrak{m} \cap W_R(M)$ is an ideal of $R$, it follows that $I \subseteq\mathfrak{m} \cap W_R(M)$.
Now, since $M$ satisfies the dual of  Property $\mathcal{A}$, we get  $IM \not=M$. Hence $M$ satisfies the dual of proper strong  Property $\mathcal{A}$.
\end{proof}

\begin{thm}\label{l021.12}
Let $M$ ba an $R$-module. Then
$$
M\\\ satisfies\\\ the\\\ dual\\\ of\\\ strong\\\ Property\\\ \mathcal{A}\\\ \Rightarrow
$$
$$
  \\\ M\\\ satisfies\\\ the\\\ dual\\\ of\\\ proper\\\ strong\\\  Property\\\ \mathcal{A}\\\  \Rightarrow
 $$
 $$
\\\ M\\\ satisfies\\\ the\\\ dual\\\ of\\\ Property\\\ \mathcal{A}.
$$
\end{thm}
\begin{proof}
The proof is clear from the definitions.
\end{proof}
The Examples \ref{e1241.1} and \ref{e402.2} show that the converse of Theorem \ref{l021.12} is not true in general.

\begin{ex}\label{e1241.1}
The $\Bbb Z$-module $\Bbb Z$ satisfies the dual of proper strong  Property $\mathcal{A}$ but does not satisfies the dual of strong Property $\mathcal{A}$.
\end{ex}

Let $R_i$ be a commutative ring with identity and  $M_i$ be an $R_i$-module for each $i = 1, 2$. Assume that
$M = M_1\times M_2$ and $R = R_1\times R_2$. Then $M$ is clearly
an $R$-module with component-wise addition and scalar multiplication. Also,
each submodule $N$ of $M$ is of the form $N = N_1\times N_2$, where $N_i$ is a
submodule of $M_i$ for each $i = 1, 2$.
\begin{prop}\label{l1.11}
Let $R_i$ be a commutative ring with identity and $M_i$ be an $R_i$-module for each $i = 1, 2$. Then
$$
W_{R_1\times  R_2}(M_1\times  M_2)=(W_{R_1}(M_1) \times R_2) \cup (R_1 \times W_{R_2}(M_2)).
$$
\end{prop}
\begin{proof}
This is straightforward.
\end{proof}

\begin{thm}\label{l2.12}
Let $R_i$ be a commutative ring with identity and $M_i$ be an $R_i$-module for each $i = 1, 2$. Let
$M = M_1\times  M_2$, $R = R_1\times  R_2$, and $S_i=R_i \setminus W_{R_i}(M_i)$. Then
the following assertions are equivalent:
\begin{itemize}
\item [(a)] $M$ satisfies the dual of proper strong  Property $\mathcal{A}$;
\item [(b)] $M_i$ satisfies the dual of proper strong Property $\mathcal{A}$  and $S_i^{-1}R_i=R_i$ for each $i = 1, 2$.
\end{itemize}
\end{thm}
\begin{proof}
$(a)\Rightarrow (b)$
Assume that $M$ satisfies the dual of proper strong  Property $\mathcal{A}$ and $I_1=\langle a_1,a_2,...,a_n\rangle$ is a finitely generated ideal of $R_1$ such that $I_1 \subseteq W_{R_1}(M_1)$. Set
 $$
 I=\langle (a_1,0), (a_2,0),...,(a_n,0), (0,1)\rangle.
 $$
 Then $(0,1), (a_i,0) \in W_{R_1 \times R_2}(M_1 \times M_2)$ for $i=1,2,...,n$. By part (a), $I(M_1\times  M_2)\not = M_1\times  M_2$. Thus there exists $(x_1,x_2) \in M_1\times  M_2 \setminus I(M_1\times  M_2)$. This implies that $x_1 \not \in I_1M_1$. Thus  $I_1M_1\not =M_1$ and $M_1$ satisfies the dual of proper strong  Property $\mathcal{A}$. Now let $r_1 \in R_1 \setminus  U(R_1) $. Clearly $(r_1, 0) \in W_{R_1 \times R_2}(M_1 \times M_2)$. Set  $
J=\langle (r_1,0), (0,1)\rangle$. Thus by part (a), $J(M_1\times  M_2)\not = M_1\times  M_2$. This implies that $r_1M_1 \not =M_1$ and hence $r_1 \in  W_{R_1}(M_1)$.
Now by Lemma \ref{l1.8},  $S_1^{-1}R_1=R_1$.   Similarly,  one can see that $M_2$ satisfies the dual of proper strong  Property $\mathcal{A}$ and $S_2^{-1}R_2=R_2$.

$(b)\Rightarrow (a)$
Let  $ I=\langle (a_1,b_1), (a_2,b_2),...,(a_n,b_n)\rangle$ be a proper finitely generated ideal of $R$ such that $(a_i,b_i) \in W_{R_1 \times R_2}(M_1 \times M_2)$ for each $i=1,2,...,n$. Set $I_1=\langle a_1,a_2,...,a_n\rangle$ and $I_2=\langle b_1,b_2,...,b_n\rangle$. Then as $I$ is proper, $I_1$ or $I_2$ is proper. Assume that $I_1$ is proper. Then by Lemma \ref{l1.8}, $I_1 \subseteq W_{R_1}(M_1)$. Thus by part (b), $I_1M_1 \not =M_1$. Hence there exists $x_1 \in M_1 \setminus I_1M_1$. Now  $(x_1,0) \in (M_1 \times M_2) \setminus I(M_1 \times M_2)$ implies that  $M$ satisfies the dual of proper strong  Property $\mathcal{A}$.
\end{proof}

\begin{thm}\label{t42.12}
Let $M$ be an $R$-module and $S=R \setminus W_R(M)$. Then we have the following.
\begin{itemize}
\item [(a)] If $S^{-1}R=R$, then $M$ satisfies the dual of proper strong Property $\mathcal{A}$ if and only if $M$ satisfies the dual of Property $\mathcal{A}$.
\item [(b)] If $S^{-1}R\not=R$, then $M$ satisfies the dual of proper strong Property $\mathcal{A}$ if and only if $M$ satisfies the dual of strong Property $\mathcal{A}$.
\end{itemize}
\end{thm}
\begin{proof}
(a) Since $S^{-1}R=R$, we have $W_R(M) =\cup_{\mathfrak{m} \in max(R)} \mathfrak{m}$ by Lemma \ref{l1.8}. Thus
far each maximal ideal $\mathfrak{m}$ of $R$, we have $\mathfrak{m} \cap W_R(M)=\mathfrak{m}$ is always an ideal of $R$. Now the result follows from Theorem \ref{t1.10}. The reverse implication is clear.

(b) If $M$ satisfies the dual of strong Property $\mathcal{A}$, then clearly,  $M$ satisfies the dual of proper strong Property $\mathcal{A}$. Conversely, assume that $M$ satisfies the dual of proper strong Property $\mathcal{A}$. As $S^{-1}R\not=R$, there exists $x \in R$ such that $x$ is not invertible and $x \not \in W_R(M)$ by Lemma \ref{l1.8}. Let $\mathfrak{m}$ be a maximal ideal of R such that $x \in m$. Let $I = \langle a_1, a_2, ..., a_n\rangle$  be a proper ideal of $R$ such that $a_i \in W_R(M)$ for $i = 1,..., n$. Then $xI = \langle xa_1,..., xa_n\rangle \subseteq  \mathfrak{m}$ is a proper ideal of $R$ and $xa_i \in W_R(M)$ for $i = 1,..., n$. Since $M$ satisfies the dual of proper strong Property $\mathcal{A}$, $xIM \not =M$. Thus there exists a completely irreducible submodule $L$ of $M$ such that $xIM\subseteq L \not =M$ by Remark \ref{r2.1}. Now, as $x \not \in W_R(M)$, it
follows that $IM\subseteq L \not =M$. This implies that $a_iM\subseteq L \not =M$  for $i = 1,..., n$. Hence  $M$ satisfies the dual of strong Property $\mathcal{A}$.
\end{proof}

\begin{prop}\label{t4332.12}
 Let $R$ be a zero-dimensional ring. Then any faithful $R$-module $M$ satisfies the dual of proper strong Property $\mathcal{A}$. In particular, any $R$-module $M$ satisfies the dual of proper strong Property $\mathcal{A}$ over $R/Ann_R(M)$.
\end{prop}
\begin{proof}
 By Theorem \ref{t1.6} (c), $M$ satisfies the dual of Property $\mathcal{A}$.
 We have $W_R(M) \subseteq Z_R(M)=Z_R(R)$ by using the proof of \cite[Corollary 2.20]{MR4160990}. Therefore,  $W_R(M)=Z_R(R)$ because the inverse inclusion is clear. Thus $S^{-1}R=R$, where $S=R \setminus W_R(M)=R \setminus Z_R(R)$. This implies that $M$ satisfies the dual of proper strong Property $\mathcal{A}$ by Theorem \ref{t42.12} (a).
\end{proof}

\begin{thm}\label{l02.12}
Let $R_i$ be a commutative ring with identity and $M_i$ be an $R_i$-module for each $i = 1, 2$. Let
$M = M_1\times  M_2$, $R = R_1\times  R_2$, and $S_i=R_i \setminus Z_{R_i}(M_i)$. Then
the following assertions are equivalent:
\begin{itemize}
\item [(a)] $M$ satisfies the proper strong  Property $\mathcal{A}$;
\item [(b)] $M_i$ satisfies the proper strong Property $\mathcal{A}$  and $S_i^{-1}R_i=R_i$ for each $i = 1, 2$.
\end{itemize}
\end{thm}
\begin{proof}
$(a)\Rightarrow (b)$
Assume that $M$ satisfies the proper strong Property $\mathcal{A}$ and $I_1=\langle a_1,a_2,...,a_n\rangle$ is a finitely generated ideal of $R_1$ such that $I_1 \subseteq Z_{R_1}(M_1)$. Set
 $$
 I=\langle (a_1,0), (a_2,0),...,(a_n,0), (0,1)\rangle.
 $$
 Then $(0,1), (a_i,0) \in Z_{R_1 \times R_2}(M_1 \times M_2)$ for $i=1,2,...,n$. By part (a), $(0:_{M_1\times  M_2}I) \not=0$. Thus there exists $0\not=(x_1,x_2) \in M_1\times  M_2$ such that $I(x_1,x_2)=0$.  This implies that $0\not=x_1 \subseteq (0:_{M_1}I_1)$ and $M_1$ satisfies the proper strong  Property $\mathcal{A}$. Now let $r_1 \in R_1 \setminus  U(R_1) $. Clearly $(r_1, 0) \in Z_{R_1 \times R_2}(M_1 \times M_2)$. Set  $
J=\langle (r_1,0), (0,1)\rangle$. Thus by part (a), $(0:_{M_1\times  M_2}J)\not =0$. This implies that  $r_1 \in  Z_{R_1}(M_1)$.
Now by  \cite[Lemma 2.1]{MR4160990},  $S_1^{-1}R_1=R_1$.   Similarly,  one can see that $M_2$ satisfies the proper strong  Property $\mathcal{A}$ and $S_2^{-1}R_2=R_2$.

$(b)\Rightarrow (a)$
Let  $ I=\langle (a_1,b_1), (a_2,b_2),...,(a_n,b_n)\rangle$ be a proper finitely generated ideal of $R$ such that $(a_i,b_i) \in Z_{R_1 \times R_2}(M_1 \times M_2)$ for each $i=1,2,...,n$. Set $I_1=\langle a_1,a_2,...,a_n\rangle$ and $I_2=\langle b_1,b_2,...,b_n\rangle$. Then as $I$ is proper, $I_1$ or $I_2$ is proper. Assume that $I_1$ is proper. Then $I_1 \subseteq Z_{R_1}(M_1)$. Thus by part (b), $(0:_{M_1}I_1) \not =0$. Hence there exists $0\not =x_1 \in M_1$ such that $I_1x_1=0$. Now  $(0,0) \not=(x_1,0) \in (0:_{M_1 \times M_2}I)$ implies that $M$ satisfies the proper strong  Property $\mathcal{A}$.
\end{proof}

The following examples show that the concepts of proper strong Property $\mathcal{A}$ and the dual of proper strong Property $\mathcal{A}$ are different in general.

\begin{ex}\label{e2.2}
Consider the $\Bbb Z\times \Bbb Z$-module $\Bbb Z \times \Bbb Z$. As $W_{\Bbb Z}(\Bbb Z)=\Bbb Z \setminus \{1, -1\}$, we have
 $S=\Bbb Z \setminus W_{\Bbb Z}(\Bbb Z)=  \{1, -1\}$ and so $S^{-1} \Bbb Z= \Bbb Z$. Now Theorem \ref{l2.12} implies that $\Bbb Z\times \Bbb Z$ satisfies the dual of proper strong Property $\mathcal{A}$. But by \cite[Example 2.13 (1)]{MR4160990},  $\Bbb Z\times \Bbb Z$ not satisfies the proper strong Property $\mathcal{A}$.
\end{ex}

\begin{ex}\label{e402.2}
Consider the $\Bbb Z\times \Bbb Z$-module $\Bbb {Q/Z}\times \Bbb {Q/Z}$. Since $W_{\Bbb Z}(\Bbb {Q/Z})=\{0\}$, we have
$S=\Bbb Z \setminus W_{\Bbb Z}(\Bbb {Q/Z})=\Bbb Z \setminus \{0\}$
 and so $\Bbb {Q}=S^{-1} \Bbb Z\not= \Bbb Z$. Now Theorem \ref{l2.12} implies that $\Bbb {Q/Z}\times \Bbb {Q/Z}$ not satisfies the dual of proper strong Property $\mathcal{A}$. On the other hand since $Z_{\Bbb Z}(\Bbb {Q/Z})=\Bbb Z \setminus \{1, -1\}$, we have
$S=\Bbb Z \setminus W_{\Bbb Z}(\Bbb {Q/Z})= \{1, -1,\}$
 and so $S^{-1} \Bbb Z= \Bbb Z$. Clearly, the $\Bbb Z$-module $\Bbb {Q/Z}$  satisfies the proper strong Property $\mathcal{A}$. Now Theorem \ref{l02.12}, implies that $\Bbb {Q/Z}\times \Bbb {Q/Z}$ satisfies the proper strong Property $\mathcal{A}$.
\end{ex}

\begin{thm}\label{p002.12}
 Let $M=R/\mathfrak{m}_1 \oplus R/\mathfrak{m}_2\oplus ...\oplus R/\mathfrak{m}_n$ be an $R$-module, where $\mathfrak{m}_i\in max(R)$ for
$i = 1, ..., n$.  Then $M$ satisfies the dual of proper strong Property $\mathcal{A}$ if and only if either $max(R) =\{\mathfrak{m}_1,\mathfrak{m}_2,...,\mathfrak{m}_n\}$ or $\mathfrak{m}_1 = \mathfrak{m}_2 = ... = \mathfrak{m}_n$.
\end{thm}
\begin{proof}
First, note that $W_R(M) =\mathfrak{m}_1\cup \mathfrak{m}_2 \cup...\cup \mathfrak{m}_n$. So, it is easy to see that any ideal $I$
contained in  $W_R(M)$ is contained in some maximal ideal $m_j$. By Theorem \ref{t1.6} (d), $M$ satisfies the dual of Property $\mathcal{A}$. Now, assume that $max(R) =\{\mathfrak{m}_1,\mathfrak{m}_2,...,\mathfrak{m}_n\}$.
Then, we get $S^{-1}R=R$. Then, using Theorem \ref{t42.12} (a),  $M$ satisfies the dual of proper strong Property $\mathcal{A}$.
On the other hand, suppose that $\mathfrak{m}_1 = \mathfrak{m}_2 = ... = \mathfrak{m}_n$. Then $W_R(M)=\mathfrak{m}$ is an ideal of $R$ and hence $M$ satisfies the dual of strong  Property $\mathcal{A}$. It follows that $M$ satisfies the dual of proper strong  Property $\mathcal{A}$.  Conversely, assume
that $\{\mathfrak{m}_1,\mathfrak{m}_2,...,\mathfrak{m}_n\} \subset max(R)$ and that $card(\{\mathfrak{m}_1,\mathfrak{m}_2,...,\mathfrak{m}_n\}) \geq 2$. Then, by Lemma \ref{l1.8}, $S^{-1}R \not=R$ as there exists a maximal ideal $\mathfrak{m} \not \subseteq \mathfrak{m}_1\cup \mathfrak{m}_2 \cup...\cup \mathfrak{m}_n=W_R(M)$ and thus there exists an
element $x \in \mathfrak{m}$ which is neither invertible nor $x \in W_R(M)$. Suppose contrary that $W_R(M)$ is an ideal. Then as $W_R(M) =\mathfrak{m}_1\cup \mathfrak{m}_2 \cup...\cup \mathfrak{m}_n$, there exists $j \in \{1,..., n\}$ such that $W_R(M)= m_j$. Hence $\mathfrak{m}_1 = \mathfrak{m}_2 = ... = \mathfrak{m}_n$, which is a contradiction. Therefore, $W_R(M)$ is not an ideal and so $M$ not satisfies the dual of strong Property $\mathcal{A}$ by Theorem \ref{tt1.5}. Therefore, by Theorem \ref{t42.12} (b), $M$ not satisfies the dual of proper strong Property $\mathcal{A}$, as needed.
\end{proof}
%%%%%%%%%%%%%%%%%%%%%%%%%%%%%%%%%%%%%%%%%%%%%%%%%%%%%%%%%%%%%%%%%%%%%%%%%%%%%%%%%%%%%%%%%%%%%%%%%%%
%%%%%%%%%%%%%%%%%%%%%%%%%%%%%%%%%%%%%%%%%%%%%%%%%%%%%%%%%%%%%%%%%%%%%%%%%%%%%%%%%%%%%%%%%%%%%%%%%%%%%%%
%%%%%%%%%%%%%%%%%%%%%%%%%%%%%%%%%%%%%%%%%%%%%%%%%%%%%%%%%%%%%%%%%%%%%%%%%%%%%%%%%%%%%%%%%%%%%%%%%%%%%%%
%%%%%%%%%%%%%%%%%%%%%%%%%%%%%%%%%%%%%%%%%%%%%%%%%%%%%%%%%%%%%%%%%%%%%%%%%%%%%%%%%%%%%%%%%%%%%%%%%%%%%%%
%%%%%%%%%%%%%%%%%%%%%%%%%%%%%%%%%%%%%%%%%%%%%%%%%%%%%%%%%%%%%%%%%%%%%%%%%%%%%%%%%%%%%%%%%%%%%%%%%%%%%%%
%%%%%%%%%%%%%%%%%%%%%%%%%%%%%%%%%%%%%%%%%%%%%%%%%%%%%%%%%%%%%%%%%%%%%%%%%%%%%%%%%%%%%%%%%%%%%%%%%%%%%%%
%%%%%%%%%%%%%%%%%%%%%%%%%%%%%%%%%%%%%%%%%%%%%%%%%%%%%%%%%%%%%%%%%%%%%%%%%%%%%%%%%%%%%%%%%%%%%%%%%%%%%%%
%%%%%%%%%%%%%%%%%%%%%%%%%%%%%%%%%%%%%%%%%%%%%%%%%%%%%%%%%%%%%%%%%%%%%%%%%%%%%%%%%%%%%%%%%%%%%%%%%%%%%%%
%%%%%%%%%%%%%%%%%%%%%%%%%%%%%%%%%%%%%%%%%%%%%%%%%%%%%%%%%%%%%%%%%%%%%%%%%%%%%%%%%%%%%%%%%%%%%%%%%%%%%%%
%%%%%%%%%%%%%%%%%%%%%%%%%%%%%%%%%%%%%%%%%%%%%%%%%%%%%%%%%%%%%%%%%%%%%%%%%%%%%%%%%%%%%%%%%%%%%%%%%%%%%%%
%%%%%%%%%%%%%%%%%%%%%%%%%%%%%%%%%%%%%%%%%%%%%%%%%%%%%%%%%%%%%%%%%%%%%%%%%%%%%%%%%%%%%%%%%%%%%%%%%%%%%%%
\section{Properties $\mathcal{S_J(N)}$ and $\mathcal{I^M_J(N)}$}
Let $J$ be an ideal of $R$ and let $N$ be a
submodule of an $R$-module $M$. Set
$$
\mathcal{S_J(N)}= \{m \in M\mid rm \in N \\\ for \\\ some \\\ r \in R-J \}.
$$
 When $J$ is a prime ideal of $R$, then $\mathcal{S_J(N)}$ is called \emph{the saturation of $N$ with respect to $J$ or $J$-closure of $N$}
\cite{MR3525784, MR1986210, MR2417474}.

Set
$$
\mathcal{I^M_J(N)}= \cap \{L \mid  L \\\ is \\\ a \\\ completely \\\
irreducible
\\\ submodule \\\ of \\\ M\\\ and
$$
$$
 rN\subseteq L \\\ for \\\ some \\\ r \in R-J \}.
$$
 When $J$ is a prime ideal of $R$, then $\mathcal{I^M_J(N)}$ is called \emph{the $J$-interior of $N$ relative to $M$}
\cite{MR3073398, MR2917107, MR3588217}.

\begin{defn}\label{d2.1}
We say that a submodule $N$ of an $R$-module $M$ satisfies \textit{Property $\mathcal{S_J(N)}$} if for each finitely generated submodule $K$ of $M$ with $K \subseteq \mathcal{S_J(N)}$
there exists a  $r \in R \setminus J$ with $rK \subseteq N$.
\end{defn}

\begin{defn}\label{d2.2}
We say that a submodule $N$ of an $R$-module $M$ satisfies \textit{Property $\mathcal{I^M_J(N)}$} (that is the dual of Property $\mathcal{S_J(N)}$) if for each  submodule  $K$ of $M$ with $M/K$ is finitely cogenerated and  $\mathcal{I^M_J(N)}\subseteq K$
there exists a $r \in R \setminus J$  with $rN \subseteq K$.
\end{defn}

\begin{defn}\label{d2.3}
We say that a submodule $N$ of an $R$-module $M$ satisfies \textit{the strong Property $\mathcal{S_J(N)}$} if for
any $m_1, . . . , m_n \in \mathcal{S_J(N)}$
there exists a $r \in R \setminus J$  with $rm_1 \in N$, ...$rm_n\in N$.
\end{defn}

\begin{ex}\label{e22.3}
Let $J$ be a prime ideal of $R$ and $N$ be a submodule of an $R$-module $M$ such that $\mathcal{S_J(N)}$ is a finitely generated submodule of $M$. Then one can see that there exists a  $r \in R \setminus J$  with $rN \subseteq \mathcal{S_J(N)}$. This implies that
$N$ satisfies the strong Property $\mathcal{S_J(N)}$.
\end{ex}

\begin{ex}\label{e23.3}
Let $J$ be a prime ideal of $R$ and $N$ be a submodule of an $R$-module $M$ such that $M/\mathcal{S_J(N)}$ is a finitely cogenerated $R$-module. Then $N$ satisfies Property $\mathcal{I^M_J(N)}$ by using \cite[Lemma 2.3]{MR3588217}.
\end{ex}

If $J=0$ and $N=0$ in Definition \ref{d2.1} (resp. Definition \ref{d2.3}), then $M$ under the name Property $T$ (strong Property $T$) was studied in \cite{MR3661610}.

\begin{thm}\label{t2.5}
Let $M$ be an $R$-module. Then we have the following.
\begin{itemize}
\item [(a)] A submodule $N$ of $M$ satisfies the strong Property $\mathcal{S_J(N)}$ if and only if $N$ satisfies Property  $\mathcal{S_J(N)}$ and  $\mathcal{S_J(N)}$ is a submodule of $M$.
\item [(b)] The zero submodule of $M$ satisfies the strong Property $\mathcal{S_J(0)}$.
\item [(c)] If a submodule $N$ of $M$ satisfies Property $\mathcal{S_J(N)}$ (respectively, strong Property $\mathcal{S_J(N)}$) and $N\subseteq K\subseteq \mathcal{S_J(N)}$, then $K$ satisfies Property $\mathcal{S_J(K)}$ (respectively,  strong Property $\mathcal{S_J(K)}$).
\item [(d)] Let $\mathfrak{\psi}=\{N_\lambda\}_ {\lambda\in \Lambda}$ be a chain of submodules of $M$ with $N \subseteq N_{\lambda} \subseteq \mathcal{S_J(N)}$ for each $\lambda\in \Lambda$. Then
    $\cup_{\lambda\in \Lambda} N_\lambda$ satisfies   Property $\mathcal{S_J(\cup_{\lambda \in \Lambda}N_\lambda)}$
(respectively,  strong Property $\mathcal{S_J(\cup_{\lambda\in \Lambda}N_\lambda)}$) if and only if each $N_\lambda$
satisfies  Property $\mathcal{S_J(N_\lambda)}$
(respectively,  strong Property $\mathcal{S_J(N_\lambda)}$).
\item [(e)] If for a submodule $N$ of $M$ we have $Ann_R(M/N)\not \subseteq J$, or more generally, $Ann_R(\mathcal{S_J(N)/N)}\not \subseteq J$, then $N$ satisfies  strong Property $\mathcal{S_J(N)}$.
\item [(f)] If $J$ is an irreducible ideal of $R$ (e.g., $R/J$ is an integral domain), then every submodule $N$ of $M$ satisfies  strong Property $\mathcal{S_J(N)}$.
\item [(g)] If a submodule $N$ of $M$ is a Bezout module (respectively, chained module), then $N$ satisfies Property
$\mathcal{S_J(N)}$ (respectively,  strong Property $\mathcal{S_J(N)}$).
\end{itemize}
\end{thm}
\begin{proof}
(a), (b), (e), and (g) are straightforward.

(c) Let $T$ be a finitely generated submodule of $M$ with $T \subseteq \mathcal{S_J(K)}$. Then  $T \subseteq \mathcal{S_J(K)} \subseteq \mathcal{S_J(\mathcal{S_J(N)})}$. Clearly,  $\mathcal{S_J(\mathcal{S_J(N)})}=\mathcal{S_J(N)}$. Hence $T\subseteq \mathcal{S_J(N)}$. Now, by assumption, there exists a  $r \in R \setminus J$ with $rT \subseteq N$ and so $rT \subseteq K$.

(d) First note that, if $\cup_{\lambda\in \Lambda} N_\lambda$ satisfies Property $\mathcal{S_J(\cup_{\lambda\in \Lambda} N_\lambda)}$
(respectively, strong Property $\mathcal{S_J(\cup_{\lambda\in \Lambda}N_\lambda)}$), then each $N_\lambda$
satisfies Property $\mathcal{S_J(N_\lambda)}$
(respectively, strong Property $\mathcal{S_J(N_\lambda)}$). Conversely, suppose that  each $N_\lambda$
satisfies Property $\mathcal{S_J(N_\lambda)}$
(respectively, strong Property $\mathcal{S_J(N_\lambda)}$) and $K$ is a finitely generated submodule
of $\mathcal{S_J(\cup_{\lambda\in \Lambda}N_\lambda)}$. Then $K$ is a submodule of  $\mathcal{S_J(N_\alpha)}$ for some $\alpha \in \Lambda$ and
hence $(N_\alpha:_RK)\not\subseteq I$ and so $(\cup_{\lambda\in \Lambda}N_\lambda:_RK)\not\subseteq I$, as desired.

(f) Let $J$ be an irreducible ideal of $R$ and $m_1, . . . , m_n\in \mathcal{S_J(N)}$, where
$r_im_i\in N$ with $r_i \in R \setminus J$ for $i=1,2,..,n$. Since $J$ is irreducible, $Rr_1\cap ... \cap Rr_n\not =J$. Hence there exists $r \in (Rr_1\cap ...\cap Rr_n) \setminus J$.  Thus $rm_i \in N$ for $i=1,2,..,n$, as needed.
\end{proof}

\begin{thm}\label{t2.6}
Let $M$ be an $R$-module. Then we have the following.
\begin{itemize}
\item [(a)] Let $J$ be a prime ideal of $R$.
 A submodule $N$ of $M$ satisfies Property $\mathcal{I^M_J(N)}$ if and only if
for any completely irreducible submodules $L_1$, ..., $L_n$ of $M$ with $\mathcal{I^M_J(N)} \subseteq L_1$, ..., $\mathcal{I^M_J(N)} \subseteq L_n$ there exists a $r \in R\setminus J$ with $rN \subseteq L_1$,..., $rN \subseteq L_n$.
\item [(b)] $M$ satisfies Property $\mathcal{I^M_\mathrm{0} (M)}$ if and only if every submodule $K$ of $M$ with $M/K$ is finitely cogenerated $M/K$  satisfies Property $\mathcal{I^{M/K}_\mathrm{0}(M/K)}$.
\item [(c)] If for a submodule $N$ of $M$ we have $Ann_R(N)\not \subseteq J$, or more generally, $Ann_R(N/\mathcal{I^M_J(N))} \not \subseteq J$, then $N$ satisfies Property $\mathcal{I^M_J(N)}$.
\item [(d)] If $J$ is an irreducible ideal of $R$ (e.g., $R/J$ is an integral domain), then every submodule $N$ of $M$ satisfies Property $\mathcal{I^M_J(N)}$.
\end{itemize}
\end{thm}
\begin{proof}
(a)
The necessity is clear. For the sufficiency assume that for a submodule  $K$ of $M$ with $M/K$ is finitely cogenerated we have $\mathcal{I^M_J(N)}\subseteq K$. As $M/K$ is finitely cogenerated, there exist completely irreducible submodules $L_1,...,L_n$ of $M$ such that
$K=\cap^n_{i=1}L_i$.  Now by assumption, there exist $r_1,...,r_n \in R \setminus J$ such that $r_iN \subseteq L_i$ for $i=1,2,..,n$.
Set $r=r_1r_2...r_n$. As $J$ is prime, $r \in R\setminus J$.  Now $rN \subseteq K$, as needed.

(b) and (c) are straightforward.

(d) Let $J$ be an irreducible ideal of $R$ and $\mathcal{I^M_J(N)} \subseteq L_i$, where  $L_i$ is a completely irreducible submodule of $M$, $r_iN\subseteq L_i$, and $r_i \in R \setminus J$ for $i=1,2,...,n$.
As $J$ is irreducible, $Rr_1\cap ...\cap Rr_n\not =J$. Hence there exists $r \in (Rr_1\cap ...\cap Rr_n) \setminus J$.  Thus $rN \subseteq L_i$ for $i=1,2,..,n$, as needed.
\end{proof}
%===========================================================================================%%

\end{document}